\documentclass[11pt]{amsart}
\usepackage[margin=1in]{geometry}

\usepackage{amssymb}
\usepackage{amsthm}
\usepackage{amsmath}
\usepackage{mathrsfs}
\usepackage{amsbsy}
\usepackage[all]{xy}
\usepackage{bm}
\usepackage{hyperref}
\usepackage{tikz}
\usepackage{array}
\usepackage{float}
\usepackage{enumerate}
\usepackage{xcolor}
\usepackage{hhline}
\allowdisplaybreaks
\usepackage{cite}
\usepackage{ragged2e}

\makeatletter
\newtheorem*{rep@theorem}{\rep@title}
\newcommand{\newreptheorem}[2]{%
\newenvironment{rep#1}[1]{%
 \def\rep@title{#2 \ref{##1}}%
 \begin{rep@theorem}}%
 {\end{rep@theorem}}}
\makeatother

\newtheorem{theorem}{Theorem}[section]
\newreptheorem{theorem}{Conjecture}
\newtheorem{lemma}{Lemma}[section]

\newtheorem{corollary}{Corollary}[section]

\newtheorem{conjecture}{Conjecture}[section]
\newtheorem{question}{Question}[section]

\theoremstyle{definition}
\newtheorem{definition}{Definition}[section]

\DeclareMathOperator{\SYT}{SYT}
\DeclareMathOperator{\ev}{ev}
\DeclareMathOperator{\rot}{rot}
\DeclareMathOperator{\comp}{comp}
\DeclareMathOperator{\rev}{rev}
\DeclareMathOperator{\Av}{Av}
\DeclareMathOperator{\SAv}{SAv}
\DeclareMathOperator{\PAv}{PAv}
\DeclareMathOperator{\ide}{id}

\begin{document}

\title[]{Pattern-Avoiding Permutation Powers} \keywords{}
\subjclass[2010]{}

\author[]{Amanda Burcroff}
\address[]{University of Michigan, 500 S. State St., Ann Arbor, MI 48109}
\email{burcroff@umich.edu}
\author[]{Colin Defant}
\address[]{Fine Hall, 304 Washington Rd., Princeton, NJ 08544}
\email{cdefant@princeton.edu}

\begin{abstract} 
Recently, B\'ona and Smith defined \emph{strong pattern avoidance}, saying that a permutation $\pi$ strongly avoids a pattern $\tau$ if $\pi$ and $\pi^2$ both avoid $\tau$. They conjectured that for every positive integer $k$, there is a permutation in $S_{k^3}$ that strongly avoids $123\cdots (k+1)$. We use the Robinson--Schensted--Knuth correspondence to settle this conjecture, showing that the number of such permutations is at least $k^{k^3/2+O(k^3/\log k)}$ and at most $k^{2k^3+O(k^3/\log k)}$. We enumerate $231$-avoiding permutations of order $3$, and we give two further enumerative results concerning strong pattern avoidance. We also consider permutations whose powers \emph{all} avoid a pattern $\tau$. Finally, we study subgroups of symmetric groups whose elements all avoid certain patterns. This leads to several new open problems connecting the group structures of symmetric groups with pattern avoidance.   
\end{abstract}
\maketitle

\section{Introduction}\label{Sec:Intro}

Consider $S_n$, the symmetric group on $n$ letters. This is the set of all permutations of the set $[n]=\{1,\ldots,n\}$, which we can view as bijections from $[n]$ to $[n]$. Many interesting questions arise when we view permutations as group elements. For example, we can ask about their cycle types, their orders, and their various powers. It is also common to view permutations as words by associating the bijection $\pi:[n]\to[n]$ with the word $\pi(1)\cdots\pi(n)$. We use this association to consider permutations as bijections and words interchangeably. This point of view allows us to consider the notion of a permutation pattern, which has spawned an enormous amount of research since its inception in the 1960's \cite{Bona, Kitaev, Linton}. Given $\pi=\pi(1)\cdots\pi(n)\in S_n$ and $\tau=\tau(1)\cdots\tau(m)\in S_m$, we say the entries $\pi(i_1),\ldots,\pi(i_m)$ form an occurrence of the pattern $\tau$ in $\pi$ if $i_1<\cdots<i_m$ and for all $j,k\in\{1,\ldots,m\}$, $\pi(i_j)<\pi(i_k)$ if and only if $\tau(j)<\tau(k)$. We say that $\pi$ \emph{avoids} $\tau$ if there does not exist an occurrence of $\tau$ in $\pi$. Let $\Av(\tau_1,\ldots,\tau_r)$ denote the set of all permutations that avoid the patterns $\tau_1,\ldots,\tau_r$. Let $\Av_n(\tau_1,\ldots,\tau_r)=\Av(\tau_1,\ldots,\tau_r)\cap S_n$. Let $\ide_n=123\cdots n$ denote the identity element of $S_n$. 

Recently, researchers have begun to explore interactions between permutation pattern avoidance and the group-theoretic properties of permutations \cite{Albert, AlbertTOTO, Archer, Atkinson1, Atkinson2, BonaCyclic, BonaSquares, Huang, Karpilovskij, Lehtonen1, Lehtonen2}. For example, a permutation is called \emph{cyclic} if it has exactly one cycle in its disjoint cycle decomposition. At the 2007 Permutation Patterns Conference, Stanley posed the problem of determining the number of cyclic permutations in $\Av_n(\tau)$. There are currently no known formulas for these numbers when $\tau$ has length at least $3$, but the articles \cite{Archer, BonaCyclic, Huang} do obtain several interesting results concerning cyclic permutations that avoid multiple patterns. 

Following a recent paper of B\'ona and Smith \cite{BonaSquares}, we say a permutation $\pi\in S_n$ \emph{strongly avoids a pattern $\tau$} if $\pi$ and $\pi^2$ both avoid $\tau$. For example, the permutation $2341$ strongly avoids $321$ because the permutations $2341$ and $(2341)^2=3412$ both avoid $321$. Let $\SAv(\tau_1,\ldots,\tau_r)$ denote the set of permutations that strongly avoid the patterns $\tau_1,\ldots,\tau_r$, and let $\SAv_n(\tau_1,\ldots,\tau_r)=\SAv(\tau_1,\ldots,\tau_r)\cap S_n$. 
B\'ona and Smith proved that $\SAv_n(\ide_{k+1})=\emptyset$ for all positive integers $n$ and $k$ with $n\geq k^3+1$. They also conjectured (see \cite[Conjecture 2.2]{BonaSquares}) that $\SAv_{k^3}(\ide_{k+1})\neq\emptyset$. The following theorem proves that this is indeed the case. In what follows, let $f^{\lambda_{p\times q}}$ denote the number of standard Young tableaux whose shape is the $p\times q$ rectangle.  

\begin{theorem}\label{Thm1}
For every positive integer $k$, we have \[\binom{\left\lfloor k^2/2\right\rfloor}{\left\lfloor k^2/4\right\rfloor}^k\left(f^{\lambda_{\lfloor k/2\rfloor\times\lceil k/2\rceil}}\right)^{2k}\leq \big|\SAv_{k^3}(\ide_{k+1})\big|\leq\left(f^{\lambda_{k^2\times k}}\right)^2.\]
\end{theorem}

The lower bound in Theorem \ref{Thm1} far surpasses the lower bound of $1$ that B\'ona and Smith conjectured. We will see in Section~\ref{Sec:Long} that 
\begin{equation}\label{Eq14}
\binom{\left\lfloor k^2/2\right\rfloor}{\left\lfloor k^2/4\right\rfloor}^k\left(f^{\lambda_{\lfloor k/2\rfloor\times\lceil k/2\rceil}}\right)^{2k}=k^{k^3/2+O(k^3/\log k)}\quad\text{and}\quad\left(f^{\lambda_{k^2\times k}}\right)^2=k^{2k^3+O(k^3/\log k)}.    
\end{equation} It is somewhat surprising that $\big|\SAv_{k^3}(\ide_{k+1})\big|$ is exponential in $k^3\log k$ since $\SAv_{k^3+1}(\ide_{k+1})$ is empty.

B\'ona and Smith \cite{BonaSquares} showed that $|\SAv_n(312)|=|\SAv_n(231)|$ and $|\SAv_n(132)|=|\SAv_n(213)|$. They determined $|\SAv_n(312)|$ exactly and found a lower bound for $|\SAv_n(321)|$. However, they were unable to compute $|\SAv_n(132)|$ explicitly. Motivated by an attempt to complete this enumeration, they asked for the number of permutations in $\Av_n(132)$ that have order $1$ or $3$. We have not answered this question, but we have proven the following theorem. Given a set $T$ of positive integers and a permutation pattern $\tau$, let $\Omega_n^T (\tau)$ denote the set of permutations $\pi\in \Av_n(\tau)$ such that the order of $\pi$ in the group $S_n$ is an element of $T$.      

\begin{theorem}\label{Thm2}
We have \[\sum_{n\geq 1}\big|\Omega_n^{\{1,3\}}(231)\big|x^n=\sum_{n\geq 1}\big|\Omega_n^{\{1,3\}}(312)\big|x^n=\frac{x+x^3+x^5}{1-x-3x^3-x^5}.\]
\end{theorem}

Additional motivation for the preceding theorem comes from the observation that the sets $\Omega_n^{\{t\}}(231)$ and $\Omega_n^{\{t\}}(312)$ have been enumerated when $t=1$ (this is trivial) and when $t=2$ (see \cite{Simion}). The next natural step concerns the enumeration of the sets $\Omega_n^{\{3\}}(231)$ and
$\Omega_n^{\{3\}}(312)$, which is immediate from Theorem \ref{Thm2}.

Although it still seems difficult to enumerate permutations that strongly avoid $\tau$ whenever $\tau\in\{132,213,321\}$, we can sometimes obtain exact formulas when we insist that our permutations strongly avoid an additional pattern of length $4$. 

\begin{theorem}\label{Thm3}
For every $n\geq 2$, we have \[|\SAv_n(132,3421)|=|\SAv_n(213,4312)|=2n^2-7n+8.\]
\end{theorem}

\begin{theorem}\label{Thm4}
We have \[1+\sum_{n\geq 1}|\SAv_n(321,3412)|x^n=\frac{1}{1-x-x^2-2x^3}.\]
\end{theorem}

As far as we can see, there is no reason to limit our attention to only permutations and their squares. It is natural to consider pattern avoidance in \emph{all} of the powers of a permutation. 

\begin{definition}\label{Def1}
We say a permutation $\pi\in S_n$ \emph{powerfully avoids a pattern $\tau$} if every power of $\pi$ avoids $\tau$. Let $\PAv(\tau_1,\ldots,\tau_r)$ denote the set of permutations that powerfully avoid the patterns $\tau_1,\ldots,\tau_r$, and let $\PAv_n(\tau_1,\ldots,\tau_r)=\PAv(\tau_1,\ldots,\tau_r)\cap S_n$. 
\end{definition}

Fix a permutation pattern $\tau$. Since powerful pattern avoidance is evidently a very stringent condition, one might naturally ask if there even exist arbitrarily long permutations that strongly avoid $\tau$. Of course, such a permutation exists if and only if $\tau$ is not an identity permutation. Indeed, if $\tau$ is not an identity permutation, then for every $n\geq 1$, the identity permutation in $S_n$ powerfully avoids $\tau$. On the other hand, if $n\geq k$, then $\PAv_n(\ide_k)$ is empty because every permutation in $S_n$ has $\ide_n$ as a power. To make this discussion nontrivial, we ask if there exist permutations \emph{of arbitrary order} that powerfully avoid $\tau$. A more refined question is as follows. What is the set of positive integers $r$ such that there exists a permutation of order $r$ that powerfully avoids $\tau$? Let us denote this set by $\Xi(\tau)$. The following theorem answers this question for most patterns $\tau$. 

\begin{theorem}\label{Thm5}
Fix $m\geq 2$ and $\tau\in S_m$. If $\tau=\ide_m$, then $\Xi(\tau)$ is equal to the set of all orders of elements of $S_{m-1}$. If $\tau\not\in\{\ide_m,m123\cdots(m-1),234\cdots m1\}$, then $\Xi(\tau)=\mathbb N$. We also have $\Xi(231)=\Xi(312)=\{1,2\}$.  
\end{theorem}

Consider the cyclic group $\mathbb Z/r\mathbb Z$. Note that $r\in\Xi(\tau)$ if and only if there is a positive integer $n$ and an injective homomorphism $\varphi: \mathbb Z/r\mathbb Z\hookrightarrow S_n$ such that $\varphi(\mathbb Z/r\mathbb Z)\subseteq\Av_n(\tau)$. This leads us to the notion of pattern avoidance in subgroups of symmetric groups. When we speak about groups, we really mean isomorphism classes of groups. This allows us to speak about sets of groups without infringing upon set-theoretic paradoxes. 

\begin{definition}\label{Def2}
Given a permutation pattern $\tau$, let $\mathcal G(\tau)$ be the set of groups $G$ such that there exists a positive integer $n$ and an injective homomorphism $\varphi: G\hookrightarrow S_n$ with $\varphi(G)\subseteq\Av_n(\tau)$. 
\end{definition}

We will obtain the following result as a simple consequence of Theorem \ref{Thm5}. See Section \ref{Sec:Preliminaries} for the definition of a sum indecomposable permutation. 

\begin{corollary}\label{Cor2}
The sets $\mathcal G(231)$ and $\mathcal G(312)$ are both equal to the set of elementary abelian $2$-groups. If $m\geq 2$ and $\tau\in S_m\setminus\{m123\cdots(m-1),234\cdots m1\}$ is sum indecomposable, then $\mathcal G(\tau)$ contains all abelian groups.  
\end{corollary}

\section{Preliminaries and Notation}\label{Sec:Preliminaries}

The \emph{plot} of a permutation $\pi=\pi(1)\cdots\pi(n)\in S_n$ is the graph displaying the points $(i,\pi(i))$ for all $i\in[n]$. The \emph{reverse} of $\pi$ is the permutation $\rev(\pi)=\pi(n)\cdots\pi(1)$. The \emph{complement} of $\pi$ is the permutation $\comp(\pi)=(n+1-\pi(1))\cdots(n+1-\pi(n))$. The \emph{inverse} of $\pi$, denoted $\pi^{-1}$, is just the inverse of $\pi$ in the group $S_n$. The plots of the reverse, complement, and inverse of $\pi$ are obtained by reflecting the plot of $\pi$ across the lines $x=\frac{n+1}{2}$, $y=\frac{n+1}{2}$, and $y=x$, respectively. We call $\comp(\rev(\pi))$ the \emph{reverse complement} of $\pi$. Let $\rot(\pi)$ be the permutation whose plot is obtained by rotating the plot of $\pi$ by $90^\circ$ counterclockwise. It is straightforward to check that $\rot(\pi)=\rev(\pi^{-1})$ and that $\rot(\rot(\pi))=\comp(\rev(\pi))$. 

Let $\delta_n=n\cdots 321$ be the reverse of the identity element $\ide_n=123\cdots n$. One can check that $\rev(\pi)=\pi\circ\delta_n$ and $\comp(\pi)=\delta_n\circ\pi$, respectively, where the symbol $\circ$ represents the product in the group $S_n$ (it is just the composition of bijections). This 
means that the reverse complement of $\pi$ is $\delta_n\circ\pi\circ\delta_n$, which is a conjugate of $\pi$ in $S_n$ because $\delta_n=\delta_n^{-1}$. It follows that $\comp(\rev(\pi^k))=\comp(\rev(\pi))^k$ for each positive integer $k$. Consequently, if $\tau_1',\ldots,\tau_r'$ are the reverse complements of the patterns $\tau_1,\ldots,\tau_r$, then \[|\Av_n(\tau_1,\ldots,\tau_r)|=|\Av_n(\tau_1',\ldots,\tau_r')|,\quad |\SAv_n(\tau_1,\ldots,\tau_r)|=|\SAv_n(\tau_1',\ldots,\tau_r')|,\] \[ \text{and}\quad |\PAv_n(\tau_1,\ldots,\tau_r)|=|\PAv_n(\tau_1',\ldots,\tau_r')|\quad\text{for every }n\geq 1.\] For example, $|\PAv_n(231)|=|\PAv_n(312)|$ and $|\PAv_n(132)|=|\PAv_n(213)|$ for every $n\geq 1$. 

Given $\sigma\in S_\ell$ and $\mu\in S_m$, let $\sigma\oplus\mu$ denote the \emph{sum} of $\sigma$ and $\mu$. This is the permutation whose plot is obtained by placing the plot of $\mu$ above and to the right of the plot of $\sigma$. The \emph{skew sum} of $\sigma$ and $\mu$, denoted $\sigma\ominus\mu$, is the permutation whose plot is obtained by placing the plot of $\mu$ below and to the right of the plot of $\sigma$. More precisely, we have \[(\sigma\oplus\mu)(i)=\begin{cases} \sigma(i) & \mbox{if } 1\leq i\leq \ell; \\ \mu(i-\ell)+\ell & \mbox{if } \ell+1\leq i\leq \ell+m \end{cases}\] and \[(\sigma\ominus\mu)(i)=\begin{cases} \sigma(i)+m & \mbox{if } 1\leq i\leq \ell; \\ \mu(i-\ell) & \mbox{if } \ell+1\leq i\leq \ell+m. \end{cases}\] 
We always have $\sigma^k\oplus\mu^k=(\sigma\oplus\mu)^k$, but the analogous statement with the sum replaced by a skew sum is false in general. A permutation is called \emph{sum indecomposable} if it cannot be written as the sum of two smaller permutations. 

We make the convention that $S_0=\{\varepsilon\}$, where $\varepsilon$ is the empty permutation. For the sake of convenience, we give $S_0$ the trivial group structure. Thus, $\varepsilon$ has order $1$. 

\section{Long Permutations that Strongly Avoid \texorpdfstring{$\ide_{k+1}$}{idk+1}}\label{Sec:Long}

We begin this section by proving the lower bound in Theorem \ref{Thm1}. Suppose $\mu_1,\ldots,\mu_k\in\Av_{k^2}(\ide_{k+1})$ are such that $\mu_{k-i+1}\circ\mu_i=\delta_{k^2}$ for all $i\in[k]$. The permutation $\pi=\mu_1\ominus\cdots\ominus\mu_k$ of length $k^3$ avoids $\ide_{k+1}$. In fact, $\pi$ strongly avoids $\ide_{k+1}$ because  \[\pi^2=\bigoplus_{i=1}^k(\mu_{k-i+1}\circ\mu_i)=\bigoplus_{i=1}^k\delta_{k^2}.\] Our goal is to determine the number of ways to choose the permutations $\mu_1,\ldots,\mu_k$ with these properties. Let \[a(k)=|\{\mu\in\Av_{k^2}(\ide_{k+1},\delta_{k+1}):\mu=\comp(\rev(\mu))\}|\] and \[ a'(k)=|\{\mu\in\Av_{k^2}(\ide_{k+1},\delta_{k+1}):\mu=\rot(\mu)\}|.\]

\begin{lemma}\label{Lem1}
Preserving the above notation, we have \[|\SAv_{k^3}(\ide_{k+1})|\geq\begin{cases} a(k)^{k/2}, & \mbox{if } k\equiv 0\pmod 2; \\ a(k)^{(k-1)/2}a'(k), & \mbox{if } k\equiv 1\pmod 2. \end{cases} \]  
\end{lemma}

\begin{proof}
We have seen that $|\SAv_{k^3}(\ide_{k+1})|$ is at least the number of ways to choose $\mu_1,\ldots,\mu_k\in\Av_{k^2}(\ide_{k+1})$ with $\mu_{k-i+1}\circ\mu_i=\delta_{k^2}$ for all $i\in[k]$. The equation $\mu_{k-i+1}\circ\mu_i=\delta_{k^2}$ shows that $\mu_{k-i+1}$ determines $\mu_i$. Indeed, it is equivalent to the equation $\mu_i=\mu_{k-i+1}^{-1}\circ\delta_{k^2}$. We have $\mu_{k-i+1}^{-1}\circ\delta_{k^2}=\rev(\mu_{k-i+1}^{-1})=\rot(\mu_{k-i+1})$, so we need $\mu_i=\rot(\mu_{k-i+1})=\rot(\rot(\mu_i))=\comp(\rev(\mu_i))$ for all $i\in[k]$.
Because the plot of each $\mu_i$ is obtained by rotating that of $\mu_{k-i+1}$ by $90^\circ$, the permutations $\mu_1,\ldots,\mu_k$ all avoid $\ide_{k+1}$ if and only if they all avoid $\ide_{k+1}$ \emph{and} $\delta_{k+1}$. If $k$ is even, then it follows that the number of ways to choose $\mu_1,\ldots,\mu_k\in\Av_{k^2}(\ide_{k+1})$ with $\mu_{k-i+1}\circ\mu_i=\delta_{k^2}$ for all $i\in[k]$ is the same as the number of ways to choose $\mu_1,\ldots,\mu_{k/2}\in\Av_{k^2}(\ide_{k+1},\delta_{k+1})$ with $\mu_i=\comp(\rev(\mu_i))$ for all $i\in[k/2]$. If $k$ is odd, then we also need $\mu_{(k+1)/2}=\rot(\mu_{(k+1)/2})$. In this case, the number of ways to choose $\mu_1,\ldots,\mu_k\in\Av_{k^2}(\ide_{k+1})$ with $\mu_{k-i+1}\circ\mu_i=\delta_{k^2}$ for all $i\in[k]$ is the same as the number of ways to choose $\mu_1,\ldots,\mu_{(k+1)/2}\in\Av_{k^2}(\ide_{k+1},\delta_{k+1})$ with $\mu_i=\comp(\rev(\mu_i))$ for all $i\in[(k-1)/2]$ and $\mu_{(k+1)/2}=\rot(\mu_{(k+1)/2})$. 
\end{proof}

At this point, we make use of the Robinson--Schensted--Knuth (RSK) correspondence. This famous bijection sends each permutation $\mu\in S_n$ to a pair $(P(\mu),Q(\mu))$ of standard Young tableaux on $n$ boxes that have the same shape. We refer the reader to \cite{Bona,Sagan} for information about the RSK correspondence. Unless otherwise stated, we assume all standard Young tableaux on $n$ boxes are filled with the elements of $[n]$. 

It is well known that the length of the first row in $P(\mu)$ (which is also the length of the first row in $Q(\mu)$ since these two tableaux have the same shape) is the length of the longest increasing subsequence of $\mu$. Similarly, the length of the first column of $P(\mu)$ is the length of the longest decreasing subsequence of $\mu$. It follows that $\mu\in\Av_{k^2}(\ide_{k+1},\delta_{k+1})$ if and only if $P(\mu),Q(\mu)\in\SYT(\lambda_{k\times k})$, where $\SYT(\lambda)$ denotes the set of standard Young tableaux of shape $\lambda$ and $\lambda_{p\times q}$ denotes the partition $(q,q,\ldots,q)$ of length $p$ (i.e., the partition whose Young diagram is a $p\times q$ rectangle). Let $f^{\lambda}=|\SYT(\lambda)|$. By the hook-length formula, we have \begin{equation}\label{Eq17}
f^{\lambda_{p\times q}}=\frac{(pq)!}{\prod_{i=1}^p\prod_{j=1}^q(i+j-1)}.
\end{equation}

In \cite{Schutzenberger}, Sch\"utzenberger defined a map called ``evacuation," which sends each standard Young tableau to a standard Young tableau of the same shape; we denote this map by $\ev$. A standard Young tableau is called \emph{self-evacuating} if it is fixed by the evacuation map. Let $\SYT^{\ev}(\lambda)$ be the set of self-evacuating standard Young tableaux of shape $\lambda$. Some of the many useful properties of the RSK correspondence (see Theorems 3.2.3, 3.6.6, and 3.9.4 in \cite{Sagan}) are the identities \[P(\rev(\mu))=P(\mu)^T,\quad Q(\rev(\mu))=\ev\left(Q(\mu)^T\right),\quad P(\mu^{-1})=Q(\mu),\quad Q(\mu^{-1})=P(\mu),\] where $Y^T$ denotes the transpose of the standard Young tableau $Y$. It follows that \begin{equation}\label{Eq8}
P(\rot(\mu))=P(\rev(\mu^{-1}))=P(\mu^{-1})^T=Q(\mu)^T  
\end{equation}
and 
\begin{equation}\label{Eq9}
Q(\rot(\mu))=Q(\rev(\mu^{-1}))=\ev\left(Q(\mu^{-1})^T\right)=\ev\left(P(\mu)^T\right).
\end{equation} 
Using the fact that the RSK map is bijective, we deduce from \eqref{Eq8} and \eqref{Eq9} that $\mu=\rot(\mu)$ if and only if $P(\mu)=Q(\mu)^T$ and $Q(\mu)=\ev(Q(\mu))$. Consequently, 
\begin{equation}\label{Eq7}
a'(k)=|\SYT^{\ev}(\lambda_{k\times k})|.
\end{equation}

The evacuation map satisfies $\ev\left(Y^T\right)=\ev(Y)^T$ for every standard Young tableau $Y$. Referring to \eqref{Eq8} and \eqref{Eq9} once again, we find that \[P(\comp(\rev(\mu)))=P(\rot(\rot(\mu)))=Q(\rot(\mu))^T=\ev\left(P(\mu)^T\right)^T=\ev(P(\mu)).\] Similarly, \[Q(\comp(\rev(\mu)))=Q(\rot(\rot(\mu)))=\ev\left(P(\rot(\mu))^T\right)=\ev(Q(\mu)).\] This shows that $\mu=\comp(\rev(\mu))$ if and only if $P(\mu)$ and $Q(\mu)$ are both self-evacuating. Therefore, 
\begin{equation}\label{Eq10}
a(k)=|\SYT^{\ev}(\lambda_{k\times k})|^2.
\end{equation} According to \eqref{Eq7} and \eqref{Eq10}, Lemma \ref{Lem1} tells us that 
\begin{equation}\label{Eq11}
|\SAv_{k^3}(\ide_{k+1})|\geq|\SYT^{\ev}(\lambda_{k\times k})|^k
\end{equation}
for every positive integer $k$. In order to complete the proof of the lower bound in Theorem \ref{Thm1}, we are left to prove the following lemma. 

\begin{lemma}\label{Lem2}
For every $k\geq 1$, the number of self-evacuating $k\times k$ standard Young tableaux is given by \[|\SYT^{\ev}(\lambda_{k\times k})|=\binom{\left\lfloor k^2/2\right\rfloor}{\left\lfloor k^2/4\right\rfloor}\left(f^{\lambda_{\lfloor k/2\rfloor\times\lceil k/2\rceil}}\right)^2.\]
\end{lemma}
\begin{proof}
Given a partition $\lambda$, let us start at the bottom left corner of the Young diagram of $\lambda$ and traverse the southeast perimeter. We write down the letter $v$ every time we traverse a vertical line segment, and we write an $h$ every time we traverse a horizontal line segment. This produces a word $w(\lambda)$ over the alphabet $\{v,h\}$. Let $w^o(\lambda)$ (respectively, $w^e(\lambda)$) be the word obtained by deleting the letters in $w(\lambda)$ in even (respectively, odd) positions and then removing any copies of the letter $v$ that come before the first $h$ and any copies of $h$ that come after the last $v$ in this new word. There are unique partitions $\lambda^o$ and $\lambda^e$ such that $w(\lambda^o)=w^o(\lambda)$ and $w(\lambda^e)=w^e(\lambda)$. For example, if we consider the partition $\lambda_{4\times 4}=(4,4,4,4)$, then we have \[w(\lambda_{4\times 4})=hhhhvvvv,\quad\text{so}\quad w^o(\lambda_{4\times 4})=w^e(\lambda_{4\times 4})=hhvv.\] This means that $\lambda_{4\times 4}^o=\lambda_{4\times 4}^e=\lambda_{2\times 2}$. More generally, we have $\lambda_{k\times k}^o=\lambda_{\lfloor k/2\rfloor\times\lceil k/2\rceil}$ and $\lambda_{k\times k}^e=\lambda_{\lceil k/2\rceil\times \lfloor k/2\rfloor}$.  

Given two partitions $\lambda_1\vdash n_1$ and $\lambda_2\vdash n_2$, let $\SYT(\lambda_1,\lambda_2)$ be the set of pairs $(X,Y)$ such that
\begin{itemize}
    \item $X$ is a standard Young tableaux of shape $\lambda_1$ whose entries form some $n_1$-element subset $\Psi$ of $[n_1+n_2]$ (here, we do not require that $\Psi=[n_1]$);
    \item $Y$ is a standard Young tableaux of shape $\lambda_2$ whose entries form the set $[n_1+n_2]\setminus\Psi$.
\end{itemize}

In Theorems 4.13 and 5.5 of \cite{Egge}, Egge showed that there is a bijection between self-evacuating tableaux of shape $\lambda$ and elements of certain sets of the form $\SYT(\lambda_1,\lambda_2)$, where $\lambda_1,\lambda_2$ depend on $\lambda$. We only need this result when $\lambda=\lambda_{k\times k}$, which is a particularly easy case. Specializing Egge's results, we find that there is a bijection \[\SYT^{\ev}(\lambda_{k\times k})\to\SYT(\lambda_{k\times k}^o,\lambda_{k\times k}^e)=\SYT\left(\lambda_{\lfloor k/2\rfloor\times\lceil k/2\rceil},\lambda_{\lceil k/2\rceil\times \lfloor k/2\rfloor}\right).\] 
The desired result follows since \[\big|\SYT\left(\lambda_{\lfloor k/2\rfloor\times\lceil k/2\rceil},\lambda_{\lceil k/2\rceil\times \lfloor k/2\rfloor}\right)\big|=\binom{2\left\lfloor k/2\right\rfloor\cdot\left\lceil k/2\right\rceil}{\left\lfloor k/2\right\rfloor\cdot\left\lceil k/2\right\rceil}f^{\lambda_{\lfloor k/2\rfloor\times\lceil k/2\rceil}}f^{\lambda_{\lceil k/2\rceil\times\lfloor k/2\rfloor}}\] \[=\binom{\left\lfloor k^2/2\right\rfloor}{\left\lfloor k^2/4\right\rfloor}\left(f^{\lambda_{\lfloor k/2\rfloor\times\lceil k/2\rceil}}\right)^2. \qedhere\]
\end{proof}

Lemma~\ref{Lem2} and the inequality in \eqref{Eq11} prove the lower bound in Theorem~\ref{Thm1}; we now prove the upper bound. The proof follows easily from the argument used to prove Theorem~2.1 in  \cite{BonaSquares}, which we reproduce here for easy reference. 

Suppose $\pi\in\SAv_{k^3}(\ide_{k+1})$.
This means that the standard Young tableau $P(\pi)$ has at most $k$ columns. Therefore, it follows from Greene's theorem \cite[Theorem 7.23.17]{Stanley} that we can color the entries of $\pi$ with at most $k$ colors so that each color class forms a decreasing subsequence of $\pi$. We may assume that the color red is used to color a maximum-length decreasing subsequence. Let the red entries be $\pi(i_1)>\cdots>\pi(i_m)$, where $i_1<\cdots<i_m$. Suppose by way of contradiction that $m\geq k^2+1$. By the pigeonhole principle, there exist $j_1<\cdots<j_{k+1}$ such that $\pi^2(i_{j_1}),\ldots,\pi^2(i_{j_{k+1}})$ are all the same color. Since $\pi(i_{j_{k+1}})<\cdots<\pi(i_{j_1})$, this means that $\pi^2(i_{j_{k+1}})>\cdots>\pi^2(i_{j_1})$. It follows that $\pi^2(i_{j_1}),\ldots,\pi^2(i_{j_{k+1}})$ form an occurrence of the pattern $\ide_{k+1}$ in $\pi^2$, which is a contradiction. We deduce that $m\leq k^2$, which means that $\pi$ does not contain a decreasing subsequence of length $k^2+1$. 

As mentioned above, the hypothesis that $\pi$ avoids $\ide_{k+1}$ guarantees that the first row in the standard Young tableau $P(\pi)$ has length at most $k$. We just saw that $\pi$ avoids $\delta_{k^2+1}$, so the first column in $P(\pi)$ has length at most $k^2$. This tableau has $k^3$ boxes, so it must be of shape $\lambda_{k^2\times k}$. Note that $Q(\pi)$ must also be of shape $\lambda_{k^2\times k}$. Because the RSK correspondence is a bijection, we find that $|\SAv_{k^3}(\ide_{k+1})|\leq\left(f^{\lambda_{k^2\times k}}\right)^2$, as desired. 

To finish this section, we derive the asymptotic estimates in \eqref{Eq14}. The Barnes $G$-function is defined on integers $n\geq 2$ by $G(n)=\prod_{j=1}^{n-2}j!$. It is known \cite[Equation A.6]{Voros} that
\begin{equation}\label{Eq15}
\log(G(n+1))=\frac{1}{2}n^2\log n-\frac{3}{4}n^2+O(n).
\end{equation} 
Furthermore, for any positive integers $p$ and $q$, we can rewrite \eqref{Eq17} as 
\begin{equation}\label{Eq16}
f^{\lambda_{p\times q}}=\frac{(pq)!G(p+1)G(q+1)}{G(p+q+1)}.    
\end{equation}
Combining Stirling's formula with \eqref{Eq15} and \eqref{Eq16} when $p=\left\lfloor k/2\right\rfloor$ and $q=\left\lceil k/2\right\rceil$ yields the estimate \[\log(f^{\lambda_{\lfloor k/2\rfloor\times\lceil k/2\rceil}})=\log(\left\lfloor k/2\right\rfloor\cdot\left\lceil k/2\right\rceil)!)+\log(G(\lfloor k/2\rfloor+1))+\log(G(\lceil k/2\rceil+1))-\log(G(k+1))\] \[=(k^2/4)\log(k^2/4)+\frac{1}{2}(k/2)^2\log(k/2)+\frac{1}{2}(k/2)^2\log(k/2)-\frac{1}{2}k^2\log k+O(k^2)\] 
\[
=(k^2/4)\log k+O(k^2),\] so \[\left(f^{\lambda_{\lfloor k/2\rfloor\times\lceil k/2\rceil}}\right)^{2k}=k^{k^3/2+O(k^3/\log k)}.\]
Another application of Stirling's formula yields \[\binom{\left\lfloor k^2/2\right\rfloor}{\left\lfloor k^2/4\right\rfloor}^k=2^{k^3/2+O(k\log k)}=k^{O(k^3/\log k)},\] so \[\binom{\left\lfloor k^2/2\right\rfloor}{\left\lfloor k^2/4\right\rfloor}^k\left(f^{\lambda_{\lfloor k/2\rfloor\times\lceil k/2\rceil}}\right)^{2k}=k^{k^3/2+O(k^3/\log k)}.\] This proves the first equality in \eqref{Eq14}. For the second equality, we again combine Stirling's formula with \eqref{Eq15} and \eqref{Eq16}, this time with $p=k^2$ and $q=k$, to find that \[\log\left(f^{\lambda_{k^2\times k}}\right)=\log((k^3)!)+\log(G(k^2+1))+\log(G(k+1))-\log(G(k^2+k+1))\] \[=k^3\log(k^3)+\left(\frac{1}{2}(k^2)^2\log(k^2)-\frac{3}{4}(k^2)^2\right)-\left(\frac{1}{2}(k^2+k)^2\log(k^2+k)-\frac{3}{4}(k^2+k)^2\right)+O(k^3)\] \[=3k^3\log k+k^4\log k-\frac{3}{4}k^4-(k^2+k)^2\log k+\frac{3}{4}(k^2+k)^2+O(k^3)=k^3\log k+O(k^3).\] The second equality in \eqref{Eq14} is now immediate.

\section{\texorpdfstring{$312$}{312}-Avoiding Permutations of Order \texorpdfstring{$1$}{1} or \texorpdfstring{$3$}{3}}\label{Sec:Order3}

Recall that $\Omega_n^{\{1,3\}}(\tau)$ denotes the set of permutations of order $1$ or $3$ that avoid the pattern $\tau$. It follows from the observations we made in Section~\ref{Sec:Preliminaries} that the elements of $\Omega_n^{\{1,3\}}(231)$ are precisely the reverse complements of the elements of $\Omega_n^{\{1,3\}}(312)$. This proves the first equality in Theorem~\ref{Thm2}. The remainder of this section is devoted to the second equality in that theorem. 

\begin{proof}[Proof of Theorem \ref{Thm2}]

Fix $n\geq 1$, and suppose $\pi\in \Omega_n^{\{1,3\}}(312)$. Let $j$ be such that $\pi(j)=1$. Because $\pi$ avoids $312$, we can write $\pi=\sigma\oplus\mu$ for some $\sigma\in \Av_j(312)$ and $\mu\in\Av_{n-j}(312)$. Since $\pi^3=\sigma^3\oplus\mu^3$, the permutations $\sigma$ and $\mu$ must have orders dividing $3$. This shows that every permutation in $\Omega_n^{\{1,3\}}(312)$ can be written uniquely as the sum of a permutation in $\bigcup_{m\geq 1}\Omega_m^{\{1,3\}}(312)$ that ends in the entry $1$ and a (possibly empty) permutation in $\bigcup_{m\geq 0}\Omega_m^{\{1,3\}}(312)$. Let $B(x)=\sum_{m\geq 1}b(m)x^m$, where $b(m)$ is the number of permutations in $\Omega_m^{\{1,3\}}(312)$ that end in the entry $1$. We have $\sum_{m\geq 1}\big|\Omega_m^{\{1,3\}}(312)\big|x^m=\left(1+\sum_{m\geq 1}\big|\Omega_m^{\{1,3\}}(312)\big|x^m\right)B(x)$, which we can rewrite as \[\sum_{m\geq 1}\big|\Omega_n^{\{1,3\}}(312)\big|x^m=\frac{B(x)}{1-B(x)}.\] From this, it is straightforward to check that the second equality in Theorem \ref{Thm2} is equivalent to the identity \[B(x)=x+\frac{x^3(1+x)^2}{1-2x^3}.\] 

We clearly have $b(1)=1$, so assume $n\geq 2$. Let $\pi\in \Omega_n^{\{1,3\}}(312)$ be a permutation with $\pi(n)=1$. Let $M$ be the smallest positive integer such that $\pi(n-M)\neq M+1$. Note that \[\pi(n-M)\not\in\{\pi(n-j):0\leq j<M\}=\{j+1:0\leq j<M\}=\{1,\ldots,M\},\] so $\pi(n-M)>M+1$. Let $k=\pi(1)$, and observe that $\pi(k)=n$ because $\pi^3(n)=n$. The points in the plot of $\pi$ that are not on the line $y=x$ can be partitioned into $3$-cycles. Every $3$-cycle in a permutation either forms an occurrence of the pattern $231$ or forms an occurrence of the pattern $312$. Since $\pi$ avoids $312$, each $3$-cycle in $\pi$ forms a $231$ pattern. In such a $231$ pattern, the two higher points are above the line $y=x$, and the lowest point is below this line (see Figure~\ref{Fig1}).

\begin{figure}[h]
\begin{center}
\includegraphics[width=.3\linewidth]{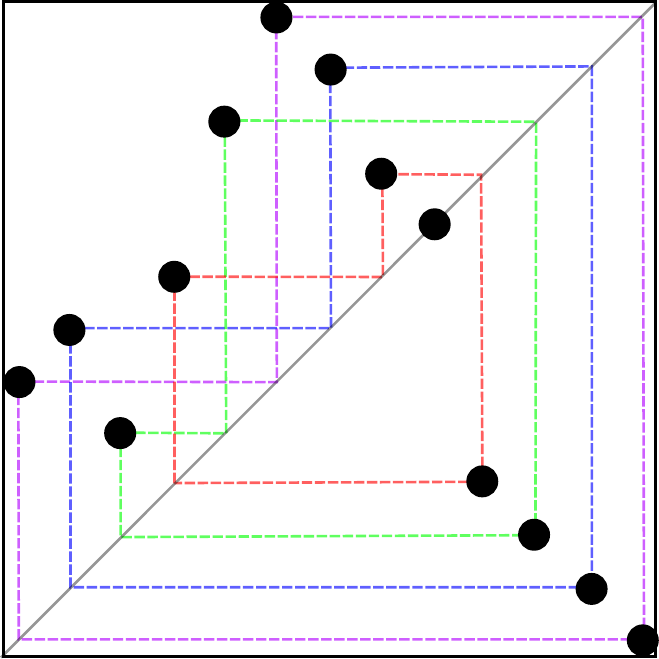}
\caption{Illustrating the proof of Theorem \ref{Thm2}, this figure shows a permutation in $\Omega_{13}^{\{1,3\}}(312)$. In this example, we have $M=4$ and $k=6$. There are four $3$-cycles, each of which is connected by a colored dotted hexagon. For example, the purple dotted hexagon indicates that $\pi(1)=6$, $\pi(6)=13$, and $\pi(13)=1$. The permutation also has $9$ as its only fixed point.}
\label{Fig1}
\end{center}  
\end{figure}

We claim that the only points below $y=x$ are the points $(i,\pi(i))$ with $n-M+1\leq i\leq n$. To see this, suppose instead that there is some $r\leq n-M$ with $\pi(r)<r$. Assume that we have chosen $r$ maximally subject to these constraints. The entries in the same $3$-cycle as $\pi(r)$ are $\pi^2(r)$ and $r$. These entries form a $231$ pattern, so $\pi(r)<\pi^2(r)<r$. If $r<k$, then the entries $k,\pi^2(r),r$ form a $312$ pattern in $\pi$. If $\pi(r)>k$, then $n,\pi^2(r),r$ form a $312$ pattern in $\pi$. Both of these situations are forbidden, so $\pi(r)<k<r$. Since $r\leq n-M$, it follows from the definition of $M$ that \[\pi(r)\not\in\{\pi(n-j):0\leq j<M\}=\{j+1:0\leq j<M\}=\{1,\ldots,M\},\] so $\pi(r)\geq M+1$. If $\pi(r)>M+1$, then the entry $M+1$ appears to the left of $\pi(r)$ by the maximality of $r$. In this case, the entries $k, M+1, \pi(r)$ form an occurrence of the pattern $312$ in $\pi$, which is impossible. Consequently, $\pi(r)=M+1$. We know that $\pi(n-M)>M+1$, so $n,M+1,\pi(n-M)$ form a $312$ pattern in $\pi$. This is our desired contradiction, so we have proven that there are only $M$ points in the plot of $\pi$ below the line $y=x$. This means that $\pi$ has exactly $M$ $3$-cycles. 

Because $\pi$ avoids $312$, there is at most one point in the plot of $\pi$ of the form $(s,s)$ with $s<k$. We claim that if such a point exists, it must be $(M+1,M+1)$. To see this, suppose instead that $s\neq M+1$. We know by the previous paragraph that $s\not\in\{1,\ldots,M\}$, so $M+1<s<k$. The previous paragraph also tells us that the point $(\pi^{-1}(M+1),M+1)$ lies above the line $y=x$, so $\pi^{-1}(M+1)<M+1<s$. This implies that the entries $k,M+1,s$ form a $312$ pattern in $\pi$, which is impossible. This proves the claim. Similarly, there is at most one point in the plot of $\pi$ of the form $(t,t)$ with $t>k$. If such a point exists, it must be $(n-M,n-M)$ because otherwise, the entries $n,t,\pi(n-M)$ would form a $312$ pattern in $\pi$. Thus, $\pi$ has at most $2$ fixed points. 

We claim that each of the sets $\{\pi(i):1\leq i\leq M\}$ and $\{\pi^{-1}(i):n-M+1\leq i\leq n\}$ is a set of consecutive integers. We will prove this claim for the first set, the proof of the claim for the second set is similar. Our proof is by induction on $n$. If $\pi(M+1)=M+1$, then removing the entry $M+1$ from $\pi$ and ``normalizing'' (that is, decrementing each entry that is greater than $M+1$ by $1$) yields a permutation $\widetilde \pi\in\Omega_{n-1}^{\{1,3\}}(312)$. By induction, the set $\{\widetilde\pi(i):1\leq i\leq M\}=\{\pi(i)-1:1\leq i\leq M\}$ is a set of consecutive integers, and this proves the claim in this case. A similar inductive argument proves the claim in the case in which $\pi(n-M)=n-M$. We saw above that $M+1$ and $n-M$ are the only possible fixed points of $\pi$, so we may now assume $\pi$ has no fixed points. It follows that every entry of $\pi$ is in a $3$-cycle and that $n=3M$.  Furthermore, $\pi(n+1-i)=i$ for all $1\leq i\leq M$, and every $3$-cycle of $\pi$ contains a unique entry in $\{1,\ldots,M\}$. If $1\leq i\leq M$ and $\pi(i)\geq n-M+1$, then $\pi^2(i)\in\{1,\ldots,M\}$, so $i$ and $\pi^2(i)$ are two elements of $\{1,\ldots,M\}$ in the same $3$-cycle. This is a contradiction, so we must have $M+1\leq \pi(i)\leq n-M=2M$ for all $1\leq i\leq M$. Hence, $\{\pi(i):1\leq i\leq M\}=\{M+1,\ldots,2M\}$ is a set of consecutive integers, as desired. 

Using the claim from the preceding paragraph, we see that if we remove the fixed points from $\pi$ and normalize the resulting permutation, then we obtain a permutation of the form $(\rot(\mu)\oplus\mu)\ominus\delta_M$, where $\mu\in\Av_M(312,213)$. Thus, specifying $\pi$ amounts to specifying $\mu$ along with which of the two possible fixed points actually appear in $\pi$. It is well known (see \cite{Linton}) that $|\Av_M(312,213)|=2^{M-1}$, so there are $2^{M-1}$ possible choices for $\mu$. In total, \[B(x)=x+\sum_{M\geq 1}2^{M-1}(x^{3M}+2x^{3M+1}+x^{3M+2})=x+(1+x)^2\sum_{M\geq 1}2^{M-1}x^{3M}=x+\frac{x^3(1+x)^2}{1-2x^3},\] as desired. 
\end{proof}

\section{Strong Pattern Avoidance Enumeration}\label{Sec:Enumeration}

The purpose of this section is to prove the enumerative results stated in Theorems \ref{Thm3} and \ref{Thm4}. 

\begin{proof}[Proof of Theorem \ref{Thm3}]
The first equality in Theorem \ref{Thm3} follows from the discussion in Section \ref{Sec:Preliminaries} because $231$ and $4312$ are the reverse complements of $132$ and $3421$, respectively. Thus, we wish to show that $|\SAv_n(132,3421)|=2n^2-7n+8$ for all $n\geq 2$. One can easily check that this is true when $n=2$ or $n=3$, so we may assume $n\geq 4$. Let us write $\SAv_n(132,3421)$ as the disjoint union $\bigcup_{i=1}^n X^{(i)}$, where $X^{(i)}=\{\pi\in\SAv_n(132,3421):\pi(i)=n\}$. There is a bijection $X^{(n)}\to\SAv_{n-1}(132,3421)$ obtained by simply removing the entry $n$ from each permutation in $X^{(n)}$. Consequently, 
\begin{equation}\label{Eq2}
\big|X^{(n)}\big|=|\SAv_{n-1}(132,3421)|.
\end{equation}

Suppose $\pi\in X^{(i)}$ for some $i\in\{2,\ldots,n-2\}$. Using the fact that $\pi$ avoids $132$ and $3421$, we find that $\pi=\sigma\ominus\ide_{n-i}$ for some $\sigma\in\Av_i(132,231)$ with $\sigma(i)=i$. Note that $\pi^2(i+1)=\pi(1)$ and $\pi^2(i+2)=\pi(2)$. Furthermore, $\pi(1)$ and $\pi(2)$ are both larger than the entry $\pi^2(i)=\pi(n)=n-i$. The entries $\pi^2(i),\pi^2(i+1),\pi^2(i+2)$ cannot form a $132$ pattern in $\pi^2$, so we must have $\pi(1)<\pi(2)$. Because $\sigma$ avoids $132$ and $231$, this forces $\sigma=\ide_i$. Therefore, $\pi=\ide_i\ominus\ide_{n-i}$. This permutation is indeed in $X^{(i)}$ because its square is either $\ide_n$ or is the skew sum of two identity permutations. Therefore, 
\begin{equation}\label{Eq3}
\big|X^{(i)}\big|=1\quad\text{for all }i\in\{2,\ldots,n-2\}.  
\end{equation}

Now assume $\pi\in X^{(n-1)}$ so that $\pi(n-1)=n$. Since $\pi$ avoids $132$, we can write $\pi=\sigma\ominus 1$, where $\sigma\in\Av_{n-1}(132,231)$ is such that $\sigma(n-1)=n-1$. Because $\sigma$ avoids $132$ and $231$, we must have either $\pi(n-2)=n-1$ or $\pi(1)=n-1$. Suppose first that $\pi(n-2)=n-1$. We have $\pi^2(n-2)=\pi(n-1)=n$ and $\pi^2(n)=\pi(1)$. Note that $\pi^2(n-1)=\pi(n)\neq\pi(2)$, so $\pi(2)$ appears to the left of $n$ in $\pi^2$. The entries $\pi(2),n,\pi(1)$ cannot form an occurrence of the pattern $132$ in $\pi^2$, so $\pi(1)<\pi(2)$. Because $\sigma$ avoids $132$ and $231$, this forces $\pi$ to be the permutation $\ide_{n-1}\ominus 1$, which is indeed an element of $X^{(n-1)}$ because its square is $\ide_{n-2}\ominus 12$. Next, assume $\pi(1)=n-1$. In this case, $\pi=((1\ominus\mu)\oplus 1)\ominus 1$ and $\pi^2=1\ominus((\mu^2\ominus 1)\oplus 1)$ for some $\mu\in S_{n-3}$. The permutation $\mu$ must be in $\SAv_{n-3}(132,231)$, and any element of this set could be $\mu$. Hence, \begin{equation}\label{Eq4}
\big|X^{(n-1)}\big|=1+|\SAv_{n-3}(132,231)|.
\end{equation}

We now want to determine $\big|X^{(1)}\big|$. For $2\leq r\leq n$, let $X^{(1)}(r)=\{\pi\in X^{(1)}:\pi(r)=1\}$. If $\pi\in X^{(1)}(n)$, then 
$\pi=1\ominus\sigma\ominus 1$ for some $\sigma\in\SAv_{n-2}(132,231)$. The permutation $\sigma$ could be any element of $\SAv_{n-2}(132,231)$, so \[\big|X^{(1)}(n)\big|=|\SAv_{n-2}(132,231)|.\] Now suppose 
$\pi\in X^{(1)}(n-1)$. Since $\pi^2(n-1)=n$ and $\pi^2$ avoids $132$, we must have $\pi^2(n)=1$. This means that $\pi(n)=n-1$. Consequently, $\pi=1\ominus((\mu\ominus 1)\oplus 1)$ for some $\mu\in S_{n-3}$. Furthermore, $\pi^2=((1\ominus\mu^2)\oplus 1))\ominus 1$, so $\mu\in\SAv_{n-3}(132,231)$. The permutation $\mu$ could be any element of $\SAv_{n-3}(132,231)$, so \[\big|X^{(1)}(n-1)\big|=|\SAv_{n-3}(132,231)|.\]
We easily check that $X^{(1)}(2)=\{n123\cdots (n-1)\}$, so $\big|X^{(1)}(2)\big|=1$. We will show that $X^{(1)}(3),\ldots,$ $X^{(1)}(n-2)$ are empty, which will imply that 
\begin{equation}\label{Eq5}
|X^{(1)}|=1+|\SAv_{n-3}(132,231)|+|\SAv_{n-2}(132,231)|.
\end{equation}

Suppose by way of contradiction that $\pi\in X^{(1)}(r)$ for some $r\in\{3,\ldots,n-2\}$. Note that $\pi^2(r)=n$. Using the fact that $\pi^2$ avoids $132$ and $3421$, one can show that $\pi^2(r+j)=j$ for all $j\in\{1,\ldots,n-r\}$. In particular, $\pi^2(r+1)=1=\pi(r)$, so $\pi(r+1)=r$. We also have $\pi(n)\leq n-1$ because $\pi(1)=n$. Observe that $r=\pi(r+1)<\pi(r+2)<\cdots<\pi(n)\leq n-1$ because $\pi$ avoids $132$ and $\pi(r)=1$. This means that $\pi(r+j)=r+j-1$ for each $j\in\{1,\ldots,n-r\}$. We deduce that $2=\pi^2(r+2)=\pi(r+1)=r$, which is our desired contradiction. 

Combining \eqref{Eq2}, \eqref{Eq3}, \eqref{Eq4}, and \eqref{Eq5}, we find that 
\begin{equation}\label{Eq1}
|\SAv_n(132,3421)|=|\SAv_{n-1}(132,3421)|+|\SAv_{n-2}(132,231)|+2|\SAv_{n-3}(132,231)|+n-1
\end{equation} when $n\geq 4$. 

We still need to determine $|\SAv_n(132,231)|$. Suppose $n\geq 4$, and let $Z^{(i)}=\{\pi\in\SAv_n(132,231):\pi(i)=n\}$. Every permutation that avoids $132$ and $231$ must either start or end in its largest entry, so $Z^{(i)}$ is empty whenever $2\leq i\leq n-1$. Removing the largest entries from the elements of $Z^{(n)}$ yields a bijection between $Z^{(n)}$ and $\SAv_{n-1}(132,231)$, so 
\begin{equation}\label{Eq6}
|Z^{(n)}|=|\SAv_{n-1}(132,231)|.
\end{equation} 
For $2\leq r\leq n$, let $Z^{(1)}(r)=\{\pi\in Z^{(1)}:\pi(r)=1\}$. Observe that $Z^{(1)}(r)\subseteq X^{(1)}(r)$, where $X^{(1)}(r)$ is as above. In particular, $Z^{(1)}(3),\ldots,Z^{(1)}(n-2)$ are empty. We also have $Z^{(1)}(2)\subseteq X^{(1)}(2)=\{n123\cdots(n-1)\}$. The square of $n123\cdots(n-1)$ contains the pattern $231$, so $Z^{(1)}(2)$ is empty. It is straightforward to check that $Z^{(1)}(n)=\{\delta_n\}$, so $|Z^{(1)}(n)|=1$. Finally, suppose $\pi\in Z^{(1)}(n-1)$. We know that $\pi\in X^{(1)}(n-1)$, so it follows from our discussion above that $\pi(n)=n-1$. Since $\pi$ avoids $132$ and $231$, we must have $\pi=n(n-2)(n-3)\cdots 321(n-1)$. However, this forces $\pi^2$ to contain $231$. This is a contradiction, so $Z^{(1)}(n-1)$ is empty. 

Putting this all together, we find that \[|\SAv_n(132,231)|=|\SAv_{n-1}(132,231)|+1\] when $n\geq 4$. It is easy to check that $|\SAv_j(132,231)|=j$ when $j\in\{1,2,3\}$. Thus, $|\SAv_n(132,231)|$ $=n$ for all $n\geq 1$. Invoking \eqref{Eq1}, we see that \[|\SAv_n(132,3421)|=|\SAv_{n-1}(132,3421)|+4n-9\] for $n\geq 4$. An easy induction now proves that \[|\SAv_n(132,3421)|=2n^2-7n+8\quad\text{for all }n\geq 2. \qedhere\] 
\end{proof}

We record the following corollary, which we demonstrated during the preceding proof. 

\begin{corollary}\label{Cor1}
For every positive integer $n$, we have \[|\SAv_n(132,231)|=n.\]
\end{corollary}

\begin{proof}[Proof of Theorem \ref{Thm4}]
Fix $n\geq 4$. Observe that since $321$ and $3412$ are involutions, the set $\SAv_n(321,3412)$ is closed under taking inverses. Let $X^{(i)}=\{\pi\in\SAv_n(321,3412):\pi(i)=n\}$ and $Y^{(i)}=\{\pi\in\SAv_n(321,3412):\pi(n)=i\}$. The elements of $X^{(i)}$ are precisely the inverses of the elements of $Y^{(i)}$. Suppose $\pi\in X^{(j)}$ for some $j\in\{1,\ldots,n-3\}$. Because $\pi$ avoids $321$, we know that $\pi(j+1)<\pi(j+2)<\cdots<\pi(n)$. Using the assumption that $\pi$ avoids $3412$, one can check that $\pi(\ell)=\ell-1$ for every $\ell\in\{j+2,\ldots,n\}$. It follows that $\pi^2(n-1)=\pi(n-2)<\pi(n-1)=\pi^2(n)=n-2$. The fact that $\pi^2(n-1)<\pi^2(n)<n-1$ forces $\pi^2$ to contain either $321$ or $3412$, which is a contradiction. We conclude that $X^{(1)},\ldots,X^{(n-3)}$ are empty. It follows that $Y^{(1)},\ldots,Y^{(n-3)}$ are also empty. 

Note that $X^{(n)}=Y^{(n)}$. We have \[\SAv_n(321,3412)=X^{(n)}\cup\bigcup_{i,j\in\{n-2,n-1\}}\left(X^{(i)}\cap Y^{(j)}\right).\]
Removing the last entries from the elements of $X^{(n)}$ yields a bijection $X^{(n)}\to\SAv_{n-1}(321,3412)$. Similarly, removing the last two entries from the elements of $X^{(n-1)}\cap Y^{(n-1)}$ yields a bijection $X^{(n-1)}\cap Y^{(n-1)}\to\SAv_{n-2}(321,3412)$. Thus, 
\begin{equation}\label{Eq12}
\big|X^{(n)}\big|=|\SAv_{n-1}(321,3412)|\quad\text{and}\quad \big|X^{(n-1)}\cap Y^{(n-1)}\big|=|\SAv_{n-2}(321,3412)|.
\end{equation} 

Every permutation $\pi\in S_n$ with $\pi(n-2)=n$ and $\pi(n)=n-2$ must contain either $321$ or $3412$. Therefore, $X^{(n-2)}\cap Y^{(n-2)}=\emptyset$. Now suppose $\pi\in X^{(n-1)}\cap Y^{(n-2)}$. We have $\pi^2(n)\leq n-1$. Since $\pi^2(n-1)=n-2$, the entry $n$ must appear to the left of $n-2$ in $\pi^2$. Because the entries $n,n-2,\pi^2(n)$ cannot form a $321$ pattern in $\pi^2$, we must have $\pi^2(n)=n-1$. This means that $\pi(n-2)=n-1$. We deduce that removing the last $3$ entries from the elements of $X^{(n-1)}\cap Y^{(n-2)}$ yields a bijection $X^{(n-1)}\cap Y^{(n-2)}\to\SAv_{n-3}(321,3412)$. The elements of $X^{(n-1)}\cap Y^{(n-2)}$ are simply the inverses of the elements of $X^{(n-2)}\cap Y^{(n-1)}$, so \begin{equation}\label{Eq13}
\big|X^{(n-1)}\cap Y^{(n-2)}\big|=\big|X^{(n-2)}\cap Y^{(n-1)}\big|=|\SAv_{n-3}(321,3412)|.
\end{equation} 
Combining \eqref{Eq12} and \eqref{Eq13} yields the equation \[|\SAv_n(321,3412)|=|\SAv_{n-1}(321,3412)|+|\SAv_{n-2}(321,3412)|+2|\SAv_{n-3}(321,3412)|\] for every $n\geq 4$. The values of $|\SAv_n(321,3412)|$ for $n=1,2,3$ are $1,2,5$. It is now routine to show that\[1+\sum_{n\geq 1}|\SAv_n(321,3412)|x^n=\frac{1}{1-x-x^2-2x^3}. \qedhere\]
\end{proof}

\section{Powerful Pattern Avoidance}\label{Sec:Powerful}
The purpose of this section is to extend our point of view to powerful pattern avoidance as well as pattern avoidance in subgroups of symmetric groups. We first prove Theorem \ref{Thm5}.

\begin{proof}[Proof of Theorem \ref{Thm5}]
Suppose $r\in\Xi(\ide_m)$. This means that there exists a positive integer $n$ and a permutation $\pi\in S_n$ of order $r$ that powerfully avoids $\ide_m$. We must have $n\leq m-1$ because $\pi^n=\ide_n$. There is an injective homomorphism $S_n\hookrightarrow S_{m-1}$, so $r$ is the order of an element of $S_{m-1}$. On the other hand, the order of a permutation in $S_{m-1}$ is in $\Xi(\ide_m)$ since that permutation must powerfully avoid $\ide_m$.

Next, note that $\Xi(231)=\Xi(312)$ because $312$ is the reverse complement of $231$. We clearly have $\{1,2\}\subseteq \Xi(231)$ since both elements of $S_2$ powerfully avoid $231$. Conversely, suppose $r\in\Xi(231)$. This means that there exists a permutation $\pi$ of order $r$ that powerfully avoids $231$. Because $\pi^{-1}$ is a power of $\pi$, the permutation $\pi^{-1}$ avoids $231$. This implies that $\pi$ avoids the inverse of $231$, which is $312$. It is well known \cite{Bona, Linton} that a permutation avoids $231$ and $312$ if and only if it is layered, meaning that it can be written in the form $\delta_{a_1}\oplus\cdots\oplus\delta_{a_t}$ for some positive integers $a_1,\ldots,a_t$. Furthermore, every layered permutation is an involution. This means that $\pi$ is an involution, so $r\in\{1,2\}$. 

Next, suppose $\tau\in S_m$ is not $\ide_m$ and is not the skew sum of two identity permutations. Choose a positive integer $r$, and consider the permutation $234\cdots r1$. This permutation has order $r$, and it powerfully avoids $\tau$ because all of its non-identity powers can be written as the skew sum of two identity permutations. It follows that $\Xi(\tau)=\mathbb N$ in this case. 

We are left to show that if $\tau=\ide_{i}\ominus\ide_j$ for integers $i,j\geq 2$, then $\Xi(\tau)=\mathbb N$. Every such permutation $\tau$ contains the pattern $3412$, so it suffices to show that $\Xi(3412)=\mathbb N$. Fix $n \in \mathbb N$, and let $m = \lfloor n/2 \rfloor$.  Let $\pi \in S_n$ be the cyclic permutation
$$\pi = (m+1)1\,2\cdots(m-1)(m+2)(m+3)\cdots n\,m.$$  
For $k \in [m]$, it is straightforward to see that the $k^{\text{th}}$ power of $\pi$ is the permutation
$$(m + k)(m + k - 1) \cdots (m + 1)1\,2\cdots (m-k)(m + k + 1)(m+k+2)\cdots (n-k+1)m(m-1)\cdots (m-k+1)$$  
and that this permutation avoids $3412$. This shows that $\pi$, $\pi^2$, \dots, $\pi^{m}$ avoid $3412$. Since $3412$ is an involution, it follows that $\pi^{-1},\pi^{-2},\ldots,\pi^{-m}$ also avoid $3412$. This shows that $\pi$ is a permutation of order $n$ that powerfully avoids $3412$, so $n \in \Xi(3412)$.
\end{proof}

We are ultimately interested in determining the sets $\mathcal G(\tau)$, which we introduced in Definition \ref{Def2}, for various patterns $\tau$. We now prove Corollary \ref{Cor2}, which makes some headway on this problem. Recall that a permutation is called sum indecomposable if it cannot be written as $\sigma\oplus\mu$ for two nonempty permutations $\sigma$ and $\mu$. Also, recall that an elementary abelian $2$-group is a group that is isomorphic to a direct product of finitely many copies of the cyclic group $\mathbb Z/2\mathbb Z$. 

\begin{proof}[Proof of Corollary \ref{Cor2}]
Our first observation in that if $\tau$ is sum indecomposable, then $\mathcal G(\tau)$ is closed under taking direct products. To see this, suppose $G_1$ and $G_2$ are in $\mathcal G(\tau)$. This means that there are positive integers $n_1,n_2$ and injective homomorphisms $\varphi_i:G_i\hookrightarrow S_{n_i}$ for $i\in\{1,2\}$ such that the elements of $\varphi_1(G_1)$ and the elements of $\varphi_2(G_2)$ avoid $\tau$. Consider the map $\varphi_1\oplus\varphi_2:G_1\times G_2\to S_{n_1+n_2}$ defined by $(\varphi_1\oplus\varphi_2)(x_1,x_2)=\varphi_1(x_1)\oplus\varphi_2(x_2)$. It is straightforward to check that $\varphi_1\oplus\varphi_2$ is an injective homomorphism. For any $x_1\in G_1$ and $x_2\in G_2$, the permutation $(\varphi_1\oplus\varphi_2)(x_1,x_2)$ avoids $\tau$. This is because $\tau$ is sum indecomposable and because $\varphi_1(x_1)$ and $\varphi_2(x_2)$ both avoid $\tau$.  

If $G$ is a group and $\varphi:G\hookrightarrow S_n$ is an injective homomorphism such that $\varphi(G)\subseteq\Av_n(231)$, then the map $\psi:G\to S_n$ given by $\psi(x)=\delta_n \varphi(x)\delta_n^{-1}$ is an injective homomorphism with $\psi(G)\subseteq\Av_n(312)$. This shows that $\mathcal G(231)\subseteq\mathcal G(312)$; the proof of the reverse containment is similar. We want to show that $\mathcal G(231)$ is the set of elementary abelian $2$-groups. The group $\mathbb Z/2\mathbb Z$ is clearly in $\mathcal G(231)$. Since $231$ is sum indecomposable, it follows from the preceding paragraph that every elementary abelian $2$-group is in $\mathcal G(231)$. On the other hand, suppose $G\in\mathcal G(231)$. It follows from Theorem \ref{Thm5} that every element of $G$ is an involution. It is well known that every finite group whose elements are all involutions is an elementary abelian $2$-group.   

Finally, suppose $m\geq 2$ and $\tau\in S_m\setminus\{m123\cdots(m-1),234\cdots m1\}$ is sum indecomposable. Theorem \ref{Thm5} tells us that $\mathcal G(\tau)$ contains every finite cyclic group. Because $\mathcal G(\tau)$ is closed under taking direct products, it follows from the Fundamental Theorem of Finitely Generated Abelian Groups that $\mathcal G(\tau)$ contains all finite abelian groups. 
\end{proof}

\section{Concluding Remarks and Open Problems}\label{Sec:Conclusion} 
In Section \ref{Sec:Long}, we proved Theorem \ref{Thm1} by constructing several permutations in $\SAv_{k^3}(\ide_{k+1})$ for every positive integer $k$. It would be interesting to gain more structural and enumerative information about the permutations in these sets. The elements we constructed all have order $4$ when $k\geq 2$, but there are also elements of $S_{k^3}(\ide_{k+1})$ of order $3$. For example, let $\zeta_{k,j}$ denote the word $(k^2-j)(2k^2-j)(3k^2-j)\cdots(k^3-j)$, and consider the concatenation $\zeta_k = \zeta_{k,0}\zeta_{k,1}\cdots\zeta_{k,k^2-1}$ forming a permutation in $\Av_{k^3}(\ide_{k+1})$. It is straightforward to verify that $\zeta_k^2$ is the permutation $\eta_k = \eta_{k,k-1}\cdots\eta_{k,0} \in \Av_{k^3}(\ide_{k+1})$, where $\eta_{k,j}$ denotes the word \[(k^3 - j)(k^3 - j - k)(k^3 - j - 2k)\cdots (k^3 - j - (k^2-1)k),\] and furthermore that $\zeta_k \circ \eta_k = \ide_{k^3}$.  Hence $\zeta_k$ and $\eta_k$ are both permutations of order $3$ in $S_{k^3}(\ide_{k+1})$. There are also elements of $\SAv_{8}(123)$ of order $12$; one such permutation is $53827614$. We are led naturally to the following question.  

\begin{question}
Does there exist a permutation $\pi\in\SAv_{k^3}(\ide_{k+1})$ of order greater than $4$ for every $k\geq 2$? More generally, what are the orders of the elements of $\SAv_{k^3}(\ide_{k+1})$? 
\end{question}

In light of Theorem \ref{Thm2}, it seems natural to investigate the sets $\Omega_n^T(\tau)$ for other patterns $\tau$ and other sets $T\subseteq\mathbb N$. There is also still much to be done in terms of enumerating sets of permutations that strongly avoid certain patterns. For instance, we have the following conjecture arising from numerical data. Let $F_n$ be the $n^\text{th}$ Fibonacci number, where we use the initial conditions $F_1=F_2=1$. 

\begin{conjecture}
For every positive integer $n$, we have $|\SAv_n(321,1342)|=2F_{n+2}-n-2$. 
\end{conjecture}

We have only scratched the surface in the study of powerful pattern avoidance and pattern avoidance in subgroups of symmetric groups. For example, Theorem \ref{Thm5} classifies the permutations that powerfully avoid $231$ (alternatively, $312$). We would like to understand the sizes of the sets $\PAv_n(132)$ and $\PAv_n(321)$. Numerical evidence suggests the following conjecture.\footnote{None of the three asymptotic equivalences implied by this conjecture are known, so it would be interesting to prove any one of them if not all of them.} 
\begin{conjecture}
We have \[\big|\Omega_n^{\{1,2,3\}}(132)\big|\sim|\PAv_n(132)|\sim|\SAv_{n-1}(132)|.\]
\end{conjecture}

The previous conjecture relates to some remarks at the end of the article by B\'ona and Smith \cite{BonaSquares}, who observed that $|\SAv_n(132)|\geq|\PAv_n(132)|\geq 2^{n(1+o(1))}$. They also indicated that the equality $|\SAv_n(132)|=2^{n(1+o(1))}$ might hold. Furthermore, they showed that $|\SAv_n(321)|\geq |\PAv_n(321)|\geq 2.3247^{\hspace{.02cm}n(1+o(1))}$. It is not clear what the correct asymptotics for $|\SAv_n(321)|$ and $|\PAv_n(321)|$ are. 

Theorem \ref{Thm5} naturally leads us to ask the following question. 

\begin{question}
What is $\Xi(2341)$? 
\end{question} 

We also have the following conjecture. Note that the conjecture is trivial when $t=1$ and that it follows from Theorem \ref{Thm5} when $t=2$.   

\begin{conjecture}
For every positive integer $t$, the set $\Xi(\ide_t\ominus 1)$ is finite. 
\end{conjecture}

Determining $\Xi(\tau)$ is equivalent to finding the set of cyclic groups in $\mathcal G(\tau)$. Ideally, we would like to know what the sets $\mathcal G(\tau)$ are in general. However, if this turns out to be too difficult, it would still be very interesting to answer the following question. Note that Corollary \ref{Cor2} answers this question for some patterns $\tau$. 

\begin{question}
Given a permutation pattern $\tau$, what is the set of all abelian groups in $\mathcal G(\tau)$? 
\end{question}

We saw in the proof of Corollary \ref{Cor2} that if $\tau$ is sum indecomposable, then $\mathcal G(\tau)$ is closed under taking direct products. Can we find other ways of constructing groups in $\mathcal G(\tau)$? 

\section{Acknowledgments}\label{Sec:Acknowledgments}
We thank Maya Sankar for a helpful discussion that led us to consider the sets $\mathcal G(\tau)$. We also thank Joe Gallian for hosting us at the University of Minnesota Duluth, where much of this research was conducted with partial support from NSF/DMS grant 1659047 and NSA grant H98230-18-1-0010. We thank the anonymous referee for helpful comments that improved the presentation of this article. The second author was supported by a Fannie and John Hertz Foundation Fellowship and an NSF Graduate Research Fellowship.

\end{document}